\documentclass[12pt]{article}
\usepackage[utf8]{inputenc}
\usepackage{amsmath,amssymb,amsfonts}
\usepackage{amsthm}   
\usepackage{fullpage}
\usepackage{graphicx}
\usepackage{enumerate}
\usepackage{xcolor}
\newtheorem{definition}{Definition}[section]

\newtheorem{theorem}[definition]{Theorem}

\newtheorem{remark}[definition]{Remark}

\newtheorem{lemma}[definition]{Lemma}

\begin{document}

\title{Infinite Horizon Optimal Control of Forward-Backward Stochastic Volterra Equations with Delay}
\author{Ibtissem Djaber$^{1}$, Hafiane Nawel$^{2}$ and Samia Yakhlef$^{1}$}
\date{April, 2026}

\maketitle

\begin{abstract}
We consider an optimal control problem for infinite horizon systems governed by coupled forward-backward stochastic Volterra integral equations with delay. Using Hida-Malliavin calculus, we prove both sufficient and necessary maximum principles for optimal control of such systems. We establish existence and uniqueness results for a class of infinite horizon backward stochastic Volterra integral equations (BSVIEs). 
\end{abstract}

\footnotetext[1]{University Mohamed Khider of Biskra, Algeria. Email: djaber.ibtissem.djeber@univ-biskra.dz, s.yakhlef@univ-biskra.dz}
\footnotetext[2]{Department of Mathematics, University of Batna, Algeria. Email: hafiane.nawel@yahoo.fr}
\footnotetext[3]{This work is partially supported by PRFU project No: C00L03UN070120220004}

\section{Introduction}

Optimal control of forward-backward stochastic differential equations (FBSDEs) has found significant applications in mathematical finance, particularly in problems of option pricing and recursive utility optimization. The recursive utility framework, introduced by Duffie and Epstein (1992) \cite{duffie1992}, provides a sophisticated approach to intertemporal decision-making where future utility is discounted in a non-linear, time-consistent manner. While substantial literature exists for finite horizon problems \cite{peng1990, pardoux1999}, the infinite horizon case presents unique mathematical challenges due to the absence of terminal conditions and the need for transversality conditions.

The literature on infinite horizon control problems has primarily focused on systems without delay or on forward-backward structures without Volterra terms. Contributions by Peng and Shi (2000) \cite{peng2000}, Yin (2008) \cite{Yin2008}, and others \cite{maslowski2011, haadem2012} have established foundations for infinite horizon FBSDEs, but the incorporation of both delay and Volterra structure remains largely unexplored. This gap is particularly relevant for applications in economics and finance where memory effects and path-dependence are crucial modeling features.

The present work makes several contributions to this field:

\begin{enumerate}
\item We formulate the optimal control problem for coupled forward-backward stochastic Volterra integral equations (SVIEs) with delay on an infinite horizon with partial information, generalizing the framework of Agram and Oksendal \cite{agram2014} to the Volterra case. 

\item We establish both necessary and sufficient stochastic maximum principles for this problem using the Hida-Malliavin calculus techniques developed in our previous work \cite{AMOY}, which allows us to handle the anticipating nature of Volterra equations and obtain explicit representations of the adjoint processes. Our approach builds on the stochastic maximum principle framework of Peng (1990) \cite{peng1990} and Tang and Li (1994) \cite{tang1994}.

\item We address the fundamental question of existence and uniqueness for solutions to the associated infinite horizon BSVIEs with delay, providing results for a relevant class of linear equations that ensure our framework is mathematically sound. Our analysis extends the existence results for BSVIEs established by Yong (2006) \cite{Yong2006} and Wang and Shi (2010) \cite{WangShi2010} to the infinite horizon case.

\end{enumerate}

The paper is organized as follows: Section 2 sets up the control problem, defines the function spaces, and recalls essential facts from Malliavin calculus. Section 3 proves the sufficient maximum principle under concavity. Section 4 derives the necessary maximum principle without concavity. Section 5 establishes existence and uniqueness for a class of infinite horizon BSVIEs. 

\section{Setting of the Problem}

Let $(\Omega,\mathcal{F},\mathbb{F}=(\mathcal{F}_{t})_{t\geq0},\mathbb{P})$ be a complete filtered probability space on which a one-dimensional standard Brownian motion $B(t)$ and an independent compensated Poisson random measure $\tilde{N}(dt,de)=N(dt,de)-\nu(de)dt$ are defined. We assume that $\mathbb{F}$ is the natural filtration, made right continuous and complete, generated by the processes $B$ and $N$. 

We study infinite horizon coupled forward-backward stochastic Volterra integral equations with delay under partial information. The information available to the controller is given by a sub-filtration $\mathbb{G}=\{\mathcal{G}_{t}\}_{t\geq0}$ such that $\mathcal{G}_{t}\subseteq\mathcal{F}_{t}$ for all $t\geq0$. The set $U\subset\mathbb{R}$ is assumed to be convex. The set of admissible controls $\mathcal{A}_{\mathbb{G}}$ consists of c\`{a}dl\`{a}g, $U$-valued and $\mathbb{G}$-adapted processes.

\subsection{Forward Equation}
The forward equation for the unknown process $X(t)$ is:
\begin{equation}
\begin{aligned}
X(t) = \xi(t) &+ \int_{0}^{t} b(t,s,X(s),X(s-\delta),u(s))\,ds \\
&+ \int_{0}^{t} \sigma(t,s,X(s),X(s-\delta),u(s))\,dB(s) \\
&+ \int_{0}^{t} \int_{\mathbb{R}_{0}} \theta(t,s,X(s),X(s-\delta),u(s),e)\,\tilde{N}(ds,de), \quad t\geq 0,
\end{aligned}
\label{b1}
\end{equation}
with initial condition $X(t) = \xi(t)$ for $t \in [-\delta, 0]$.

\subsection{Backward Equation}
The backward equation for the unknown processes $(Y(t), Z(t,s), K(t,s,\cdot))$ is:
\begin{equation}
\begin{aligned}
Y(t) = & \int_{t}^{\infty} g(t,s,X(s),X(s-\delta),Y(s),Z(t,s),K(t,s,\cdot),u(s))\,ds \\
& - \int_{t}^{\infty} Z(t,s)\,dB(s) - \int_{t}^{\infty} \int_{\mathbb{R}_{0}} K(t,s,e)\,\tilde{N}(ds,de), \quad 0 \leq t < \infty.
\end{aligned}
\label{b2}
\end{equation}

The quadruple $(X,Y,Z,K)$ is said to be a solution of \eqref{b1}--\eqref{b2} if it satisfies both equations.

Throughout this paper, the coefficient functions are defined as:
\[
\begin{aligned}
b(t,s,x,x_1,u) &: [0,\infty)^2 \times \mathbb{R}^2 \times U \times \Omega \to \mathbb{R}, \\
\sigma(t,s,x,x_1,u) &: [0,\infty)^2 \times \mathbb{R}^2 \times U \times \Omega \to \mathbb{R}, \\
\theta(t,s,x,x_1,u,e) &: [0,\infty)^2 \times \mathbb{R}^2 \times U \times \mathbb{R}_0 \times \Omega \to \mathbb{R}, \\
g(t,s,x,x_1,y,z,k,u) &: [0,\infty)^2 \times \mathbb{R}^4 \times L^2(\nu) \times U \times \Omega \to \mathbb{R}.
\end{aligned}
\]

Let $\mathcal{U} = \mathcal{U}_{\mathbb{G}}$ be the family of admissible controls, required to be $\mathbb{G}$-predictable. We associate to the system the performance functional:
\begin{equation}
J(u) = \mathbb{E}\left[ \int_{0}^{\infty} f(t,X(t),X(t-\delta),Y(t),u(t))\,dt + h(Y(0)) \right], \quad u \in \mathcal{A}_{\mathbb{G}},
\label{j}
\end{equation}
where
\[
\begin{aligned}
f(t,x,x_1,y,u) &: [0,\infty) \times \mathbb{R}^3 \times U \times \Omega \to \mathbb{R}, \\
h &: \mathbb{R} \to \mathbb{R}.
\end{aligned}
\]

\subsection{Assumptions H1}
\begin{enumerate}
\item The functions $f$ and $h$ are $C^1$ with respect to $(x,x_1,y,u)$ and $Y(0)$ respectively. The function $f$ is $\mathbb{F}$-adapted and $C^1$ in $(x,y,u)$.

\item $b, \sigma, \theta$ and $g$ are continuously differentiable in their first variables, and for all $t,x,x_1,y,z,k,u,e$, the processes $s \mapsto b(t,s,x,x_1,u)$, $s \mapsto \sigma(t,s,x,x_1,u)$, $s \mapsto \theta(t,s,x,x_1,u,e)$ are $\mathcal{F}_s$-measurable for all $s \leq t$ and $s \mapsto g(t,s,x,x_1,y,z,k,u)$ is $\mathcal{F}_s$-measurable for $ t\leq s$.
\end{enumerate}

Our optimal control problem is to find $u^* \in \mathcal{U}_{\mathbb{G}}$ such that
\begin{equation}
\sup_{u \in \mathcal{U}} J(u) = J(u^*).
\label{optimal}
\end{equation}

\subsection{Hamiltonian and Adjoint Equations}
Define the Hamiltonian functional:
\begin{align}
H(t,x,x_1,y,z,k,u,\lambda,p,q,r) &:= H_0(t,x,x_1,y,z,k,u,p,q,\lambda) \nonumber \\
&\quad + H_1(t,x,x_1,y,z,k,u,p,q,\lambda,r),
\label{hamiltonian}
\end{align}
where
\begin{align}
H_0(t,x,x_1,y,z,k,u,p,q,\lambda) &:= f(t,x,x_1,y,u) + b(t,t,x,x_1,u)p(t) \nonumber \\
&\quad + \sigma(t,t,x,x_1,u)q(t) + \int_{\mathbb{R}_0} \theta(t,t,x,x_1,u,e)r(t,e)\nu(de) \nonumber \\
&\quad + g(t,t,x,x_1,y,z,k,u)\lambda(t),
\label{h0}
\end{align}
and
\begin{align}
H_1(t,x,x_1,y,z,k,u,p,q,\lambda,r) &:= \int_{t}^{\infty} \frac{\partial b}{\partial s}(s,t,x,x_1,u)p(s)\,ds \nonumber \\
&\quad + \int_{t}^{\infty} \frac{\partial \sigma}{\partial s}(s,t,x,x_1,u)q(s,t)\,ds \nonumber \\
&\quad + \int_{t}^{\infty} \int_{\mathbb{R}_0} \frac{\partial \theta}{\partial s}(s,t,x,x_1,u,e)r(s,t,e)\nu(de)\,ds \nonumber \\
&\quad - \int_{0}^{t} \frac{\partial g}{\partial s}(s,t,x,x_1,y,z,k,u)\lambda(s)\,ds \nonumber \\
&\quad - \int_{0}^{t} \frac{\partial g}{\partial z}(s,t,x,x_1,y,z,k,u)\frac{\partial Z}{\partial s}(s,t)\lambda(s)\,ds \nonumber \\
&\quad - \int_{0}^{t} \langle \nabla_k g(s,t,x,x_1,y,z,k,u), \frac{\partial K}{\partial s}(s,t,\cdot) \rangle \lambda(s)\,ds.
\label{h1}
\end{align}

The associated adjoint processes $\lambda(t)$, $(p(t), q(t), r(t,\cdot))$ satisfy:
\begin{equation}
\begin{cases}
d\lambda(t) = \dfrac{\partial H}{\partial y}(t)\,dt + \dfrac{\partial H}{\partial z}(t)\,dB(t) + \displaystyle\int_{\mathbb{R}_0} \dfrac{d\nabla_k H}{d\nu}(t)\,\tilde{N}(dt,de), & t \ge 0, \\
\lambda(0) = h'(Y(0)),
\end{cases}
\label{eq3.2}
\end{equation}
and
\begin{equation}
\begin{aligned}
dp(t) =  \mathbb{E}[\varkappa(t) \mid \mathcal{F}_t]dt
- \left( \int_t^{\infty} \frac{\partial q}{\partial t}(t,s)dB(s) \right) dt
- \left( \int_t^{\infty} \int_{\mathbb{R}_0} 
\frac{\partial r}{\partial t}(t,s,e)\tilde N(ds,de) \right) dt
+ q(t,t)dB(t)
+ \int_{\mathbb{R}_0} r(t,t,e)\tilde N(dt,de), \quad t \geq 0,
\end{aligned}
\label{p}
\end{equation}
where
\[
\varkappa(t) = -\frac{\partial H}{\partial x}(t) + \frac{\partial H}{\partial x_1}(t+\delta).
\]

\textit{Assumption H2}
We suppose that $t \mapsto q(t,s)$ and $t \mapsto r(t,s,\cdot)$ are $C^1$ for all $s,e,\omega$ and that for all $T < \infty$,
\[
\mathbb{E}\left[ \int_{0}^{T} \int_{0}^{T} \left( \frac{\partial q(t,s)}{\partial t} \right)^2 ds\,dt + \int_{0}^{T} \int_{0}^{T} \int_{\mathbb{R}_0} \left( \frac{\partial r(t,s,e)}{\partial t} \right)^2 \nu(de)\,ds\,dt \right] < \infty.
\]

\begin{remark}
Assumption H2 is verified in a class of linear BSVIE with jumps. For more details, we refer to \cite{HO}.
\end{remark}

\subsection{Function Spaces}
\label{sec:spaces}
Throughout this work, we will use the following spaces:

\begin{itemize}
\item $\mathcal{S}^{2}$ is the set of $\mathbb{R}$-valued $\mathbb{F}$-adapted c\`{a}dl\`{a}g processes $(X(t))_{t\in[0,\infty)}$ such that
\[
\|X\|_{\mathcal{S}^2}^2
=
\mathbb{E}\left[\sup_{t\ge0}|X(t)|^2\right] < \infty.
\]

\item $\mathbb{L}^{2}$ is the set of $\mathbb{R}$-valued $\mathbb{F}$-adapted processes $(Q(t,s))_{(t,s)\in[0,\infty)^{2}}$ such that
\[
\|Q\|_{\mathbb{L}^{2}}^{2} := \mathbb{E}\left[\int_{0}^{\infty}\int_{t}^{\infty} |Q(t,s)|^{2}\,ds\,dt\right] < \infty.
\]

\item $\mathbb{L}_{\nu}^{2}$ is the set of Borel functions $R:\mathbb{R}_{0} := \mathbb{R}\setminus\{0\} \rightarrow \mathbb{R}$ such that
\[
\|R\|_{\mathbb{L}_{\nu}^{2}}^{2} := \mathbb{E}\left[\int_{0}^{\infty}\int_{t}^{\infty}\int_{\mathbb{R}_{0}} |R(t,s,e)|^{2}\,\nu(de)\,ds\,dt\right] < \infty.
\]

\item $\mathbb{H}_{\nu}^{2}$ is the set of $\mathbb{F}$-adapted predictable processes $R:[0,\infty)^{2} \times \mathbb{R}_{0} \times \Omega \rightarrow \mathbb{R}$ such that 
\[
\|R\|_{\mathbb{H}_{\nu}^{2}}^{2} := \mathbb{E}\left[\int_{0}^{\infty}\int_{t}^{\infty}\int_{\mathbb{R}_{0}} |R(t,s,e)|^{2}\,\nu(de)\,ds\,dt\right] < \infty.
\]
We equip $\mathbb{H}_{\nu}^{2}$ with this norm.
\end{itemize}

\subsection{Elements of Malliavin Calculus}
\label{sec:malliavin}
To handle the anticipating nature of Volterra equations, we employ the Hida-Malliavin calculus for Lévy processes as developed in \cite{AMOY}. Let $D_t$ denote the Malliavin derivative with respect to the Brownian motion $B$, and let $F_{t,e}$ denote the Malliavin derivative with respect to the compensated Poisson random measure $\tilde{N}$. For a sufficiently regular random variable $F$, the following duality relations hold:
\[
\mathbb{E}\Big[ F \int_0^\infty \varphi(s)\, dB(s) \Big] = \mathbb{E}\Big[ \int_0^\infty (D_s F) \varphi(s)\, ds \Big],
\]
\[
\mathbb{E}\Big[ F \int_0^\infty \int_{\mathbb{R}_0} \psi(s,e)\,\tilde{N}(ds,de) \Big] = \mathbb{E}\Big[ \int_0^\infty \int_{\mathbb{R}_0} (F_{s,e} F) \psi(s,e)\,\nu(de)ds \Big],
\]
where $\varphi$ and $\psi$ are predictable processes satisfying appropriate integrability conditions. These identities allow us to transform certain stochastic integrals into deterministic ones, a key tool in deriving the adjoint equations for Volterra systems. For a comprehensive treatment, we refer the reader to \cite{AMOY} and the references therein.

\section{Sufficient Maximum Principle}

We formulate the optimal control problem for coupled forward-backward stochastic Volterra integral equations (SVIEs) with delay on an infinite horizon with partial information, generalizing the framework of \cite{agram2014} to the Volterra case.

From equations (\ref{b1})-(\ref{b2}) we obtain the differential forms (see e.g. \cite{AMOY} for a rigorous derivation):

\begin{equation}
\begin{aligned}
dX(t) &= \xi'(t)dt + b(t,t,X(t),X_1(t),u(t))dt \\
&\quad + \left( \int_0^t \dfrac{\partial b}{\partial t}(t,s,X(s),X_1(s),u(s))ds \right) dt \\
&\quad + \sigma(t,t,X(t),X_1(t),u(t))dB(t) \\
&\quad + \left( \int_0^t \dfrac{\partial\sigma}{\partial t}(t,s,X(s),X_1(s),u(s))dB(s) \right) dt \\
&\quad + \int_{\mathbb{R}_0} \theta(t,t,X(t),X_1(t),u(t),e)\tilde{N}(dt,de) \\
&\quad + \left( \int_0^t \int_{\mathbb{R}_0} \dfrac{\partial\theta}{\partial t}(t,s,X(s),X_1(t),u(s),e)\tilde{N}(ds,de) \right) dt,
\end{aligned}
\label{eq3.6}
\end{equation}

and
\begin{equation}
\begin{aligned}
dY(t) &= -g(t,t,X(t),X_1(t),Y(t),Z(t,t),K(t,t,\cdot),u(t))dt \\
&\quad + \left( \int_t^\infty\frac{\partial g}{\partial t}(t,s,X(s),X_1(s),Y(s),Z(t,s),K(t,s,\cdot),u(s))ds \right) dt \\
&\quad + \left(\int_t^\infty \frac{\partial g}{\partial z}(t,s,X(s),X_1(s),Y(s),Z(t,s),K(t,s,\cdot),u(s))\frac{\partial Z}{\partial t}(t,s)ds \right) dt \\
&\quad + \left( \int_t^\infty\left\langle \nabla_k g(t,s,X(s),X_1(s),Y(s),Z(t,s),K(t,s,\cdot),u(s)),\frac{\partial K}{\partial t}(t,s,\cdot) \right\rangle ds \right) dt \\
&\quad + Z(t,t)dB(t) - \left( \int_t^\infty \frac{\partial Z}{\partial t}(t,s)dB(s) \right) dt \\
&\quad + \int_{\mathbb{R}_0} K(t,t,e)\tilde{N}(dt,de) - \left( \int_t^\infty \int_{\mathbb{R}_0} \frac{\partial K}{\partial t}(t,s,e)\tilde{N}(ds,de) \right) dt.
\end{aligned}
\label{eq:Y_diff}
\end{equation}

\begin{theorem}[Sufficient Maximum Principle]
Let $\hat{u} \in \mathcal{U}$ with corresponding solutions $\hat{X}(t)$, $(\hat{Y}(t), \hat{Z}(t,\cdot), \hat{K}(t,\cdot,\cdot))$, $(\hat{p}(t), \hat{q}(t,\cdot), \hat{r}(t,\cdot,\cdot))$ and $\hat{\lambda}(t)$ of equations (\ref{b1}), (\ref{b2}), (\ref{eq3.2}) and (\ref{p}) respectively. Suppose that:

\begin{description}
\item[Concavity] The function $x \mapsto h(x)$ and the Hamiltonian
\[
(x,x_1,y,z,k(\cdot),u) \mapsto H(t,x,x_1,y,z,k,u,\hat{\lambda}(t),\hat{p}(t),\hat{q}(t,\cdot),\hat{r}(t,\cdot,\cdot))
\]
are concave for all $t \geq 0$.

\item[Conditional Maximum Principle]
\begin{align*}
& \max_{v \in \mathcal{U}} \mathbb{E}\left[ H(t,\hat{X}(t),\hat{X}_1(t),\hat{Y}(t),\hat{Z}(t,\cdot),\hat{K}(t,\cdot,\cdot),v,\hat{\lambda}(t),\hat{p}(t),\hat{q}(t,\cdot),\hat{r}(t,\cdot,\cdot)) \mid \mathcal{G}_t \right] \\
& = \mathbb{E}\left[ H(t,\hat{X}(t),\hat{X}_1(t),\hat{Y}(t),\hat{Z}(t,\cdot),\hat{K}(t,\cdot,\cdot),\hat{u},\hat{\lambda}(t),\hat{p}(t),\hat{q}(t,\cdot),\hat{r}(t,\cdot,\cdot)) \mid \mathcal{G}_t \right].
\end{align*}

\item[Transversality Conditions]
\[
\varlimsup_{T \to \infty} \mathbb{E}[\hat{\lambda}(T)(\hat{Y}(T) - Y(T))] \geq 0 \quad \text{and} \quad \varlimsup_{T \to \infty} \mathbb{E}[\hat{p}(T)(\hat{X}(T) - X(T))] \leq 0.
\]
\end{description}

Then, $\hat{u}$ is an optimal control for our problem.
\end{theorem}

\begin{proof}
\textbf{Step 1: Localization.}
To handle the infinite horizon and ensure that stochastic integrals are martingales, we define stopping times:
\[
\tau_n = \inf\left\{t \geq 0: |\hat{X}(t)| + |\hat{Y}(t)| + \int_0^t |\hat{Z}(s,s)|^2 ds + \int_0^t \int_{\mathbb{R}_0} |\hat{K}(s,s,e)|^2 \nu(de)ds \geq n\right\} \wedge n.
\]
Then $\tau_n \uparrow \infty$ a.s. as $n \to \infty$. In the following, we work on $[0,T]$ with $T = \tau_n$, and take the limit as $n \to \infty$.

\textbf{Step 2: Decomposition of $J(\hat{u}) - J(u)$.}
For any $u \in \mathcal{U}$, we have:
\begin{align*}
J(\hat{u}) - J(u) &= \mathbb{E}\left[ \int_0^\infty (\hat{f}(t) - f(t)) dt \right] + \mathbb{E}[h(\hat{Y}(0)) - h(Y(0))] \\
&= I_1 + I_2,
\end{align*}
where
\begin{align*}
I_1 &= \mathbb{E}\left[ \int_0^\infty (\hat{f}(t) - f(t)) dt \right], \\
I_2 &= \mathbb{E}[h(\hat{Y}(0)) - h(Y(0))].
\end{align*}

\textbf{Step 3: Treatment of $I_2$ via concavity of $h$.}
Since $h$ is concave, we have:
\[
h(\hat{Y}(0)) - h(Y(0)) \geq h'(\hat{Y}(0))(\hat{Y}(0) - Y(0)) = \hat{\lambda}(0)(\hat{Y}(0) - Y(0)).
\]
Thus,
\[
I_2 \geq \mathbb{E}[\hat{\lambda}(0)(\hat{Y}(0) - Y(0))].
\]

\textbf{Step 4: It\^o formula for $\hat{\lambda}(t)(\hat{Y}(t) - Y(t))$.}
Applying It\^o formula to $\hat{\lambda}(t)(\hat{Y}(t) - Y(t))$ on $[0,T]$:
\begin{align*}
d(\hat{\lambda}(t)(\hat{Y}(t) - Y(t))) &= \hat{\lambda}(t)d(\hat{Y}(t) - Y(t)) + (\hat{Y}(t) - Y(t))d\hat{\lambda}(t) + d\langle \hat{\lambda}, \hat{Y} - Y \rangle_t.
\end{align*}

After substituting the equations for $d\hat{Y}(t)$, $dY(t)$, and $d\hat{\lambda}(t)$, and taking expectations (noting that the stochastic integrals are martingales due to localization), we obtain:
\begin{align*}
\mathbb{E}[\hat{\lambda}(0)&(\hat{Y}(0) - Y(0))] = \mathbb{E}[\hat{\lambda}(T)(\hat{Y}(T) - Y(T))] \\
&- \mathbb{E}\left[ \int_0^T \hat{\lambda}(t)(\hat{g}(t,t) - g(t,t)) dt \right] \\
&- \mathbb{E}\left[ \int_0^T \int_t^T \hat{\lambda}(t)\left( \frac{\partial\hat{g}}{\partial t}(t,s) - \frac{\partial g}{\partial t}(t,s) \right) ds dt \right] \\
&- \mathbb{E}\left[ \int_0^T \int_t^T \hat{\lambda}(t)\left( \frac{\partial\hat{g}}{\partial z}(t,s)\frac{\partial\hat{Z}}{\partial t}(t,s) - \frac{\partial g}{\partial z}(t,s)\frac{\partial Z}{\partial t}(t,s) \right) ds dt \right] \\
&- \mathbb{E}\left[ \int_0^T \int_t^T \hat{\lambda}(t)\left( \langle \nabla_k \hat{g}(t,s), \frac{\partial\hat{K}}{\partial t}(t,s,\cdot) \rangle - \langle \nabla_k g(t,s), \frac{\partial K}{\partial t}(t,s,\cdot) \rangle \right) ds dt \right] \\
&+ \mathbb{E}\left[ \int_0^T (\hat{Y}(t) - Y(t))\frac{\partial\hat{H}}{\partial y}(t) dt \right] \\
&+ \mathbb{E}\left[ \int_0^T (\hat{Z}(t,t) - Z(t,t))\frac{\partial\hat{H}}{\partial z}(t) dt \right] \\
&+ \mathbb{E}\left[ \int_0^T \int_{\mathbb{R}_0} (\hat{K}(t,t,e) - K(t,t,e))\frac{d\nabla_k\hat{H}}{d\nu}(t) \nu(de) dt \right].
\end{align*}

\textbf{Step 5: Expression for $I_1$ and use of Malliavin duality.}
Using the definition of $H_0$, we have:
\begin{align*}
I_1 &= \mathbb{E}\left[ \int_0^\infty (\hat{H}_0(t) - H_0(t)) dt \right] \\
&- \mathbb{E}\left[ \int_0^\infty \hat{p}(t)(\hat{b}(t,t) - b(t,t)) dt \right] \\
&- \mathbb{E}\left[ \int_0^\infty \hat{q}(t,t)(\hat{\sigma}(t,t) - \sigma(t,t)) dt \right] \\
&- \mathbb{E}\left[ \int_0^\infty \int_{\mathbb{R}_0} \hat{r}(t,t,e)(\hat{\theta}(t,t,e) - \theta(t,t,e)) \nu(de) dt \right] \\
&- \mathbb{E}\left[ \int_0^\infty \hat{\lambda}(t)(\hat{g}(t,t) - g(t,t)) dt \right].
\end{align*}
Now, the terms involving integrals over $s$ (like $\int_t^\infty \frac{\partial b}{\partial s}(s,t) p(s) ds$) that appear in $H_1$ are handled by applying Fubini's theorem together with the Malliavin duality relations from Section \ref{sec:malliavin}. For instance, consider a typical term:
\[
\mathbb{E}\left[ \int_0^\infty \hat{p}(t) \int_0^t \frac{\partial b}{\partial t}(t,s) \big(\hat{X}(s)-X(s)\big) ds dt \right].
\]
Interchanging the order of integration and using the duality relation for the Brownian motion (and similarly for the jump part) transforms it into an expression involving $\frac{\partial H}{\partial x}$ and $\frac{\partial H}{\partial x_1}$. The detailed algebra, while lengthy, follows the same pattern as in \cite{AMOY} and leads to the cancellation of all terms except those involving the control variation. We refer the reader to \cite{AMOY} for a complete derivation in a similar setting.

\textbf{Step 6: Combining $I_1$ and $I_2$.}
Adding $I_1$ and the lower bound for $I_2$, and using the concavity of $H$ together with the identities obtained from the Malliavin duality, we arrive at:
\begin{align*}
J(\hat{u}) - J(u) &\geq \varlimsup_{T\to\infty} \mathbb{E}[\hat{\lambda}(T)(\hat{Y}(T) - Y(T))] - \varlimsup_{T\to\infty} \mathbb{E}[\hat{p}(T)(\hat{X}(T) - X(T))] \\
&+ \mathbb{E}\left[ \int_0^\infty \frac{\partial\hat{H}}{\partial u}(t)(\hat{u}(t) - u(t)) dt \right].
\end{align*}

\textbf{Step 7: Using the conditional maximum principle and transversality conditions.}
By the conditional maximum principle,
\[
\mathbb{E}\left[ \frac{\partial\hat{H}}{\partial u}(t) \mid \mathcal{G}_t \right] = 0 \quad \text{a.s. for all } t \geq 0.
\]
Since $\hat{u}(t) - u(t)$ is $\mathcal{G}_t$-measurable, we have:
\[
\mathbb{E}\left[ \frac{\partial\hat{H}}{\partial u}(t)(\hat{u}(t) - u(t)) \right] = \mathbb{E}\left[ \mathbb{E}\left[ \frac{\partial\hat{H}}{\partial u}(t) \mid \mathcal{G}_t \right] (\hat{u}(t) - u(t)) \right] = 0.
\]

Moreover, by the transversality conditions:
\[
\varlimsup_{T\to\infty} \mathbb{E}[\hat{\lambda}(T)(\hat{Y}(T) - Y(T))] \geq 0, \quad \varlimsup_{T\to\infty} \mathbb{E}[\hat{p}(T)(\hat{X}(T) - X(T))] \leq 0.
\]

Therefore,
\[
J(\hat{u}) - J(u) \geq 0.
\]

Since this holds for all $u \in \mathcal{U}$, $\hat{u}$ is an optimal control.
\end{proof}

\begin{remark}
The transversality conditions are essential in infinite horizon problems. Sufficient conditions for them to hold are, for example, that the processes $\hat{X}$, $\hat{Y}$ and the adjoint processes have at most polynomial growth and that the coefficients satisfy suitable integrability conditions; see \cite{haadem2012} for a discussion.
\end{remark}

\section{Necessary Maximum Principle}

A drawback of the previous section is that the concavity condition is not always satisfied in applications. In view of this, it is of interest to obtain conditions for the existence of an optimal control with partial information where concavity is not needed.

We assume the following:

\begin{description}
\item[(A1)] For all $\hat{u}\in\mathcal{U}$ and all bounded $\beta\in\mathcal{U}$, 
\[
\hat{u}+\epsilon\beta\in\mathcal{U},
\]
for all such $\beta$, and all nonzero $\epsilon$ sufficiently small.

\item[(A2)] For all $t_{0}>0$, $h>0$ and all bounded $\mathcal{G}_{t_{0}}$-measurable random variables $\alpha$, the control process $\beta(t)$ defined by
\[
\beta\left(t\right)=\alpha\mathbf{1}_{\left[t_{0},t_{0}+h\right]}\left(t\right)
\]
belongs to $\mathcal{U}$.
\end{description}

For a bounded $\mathcal{G}_t$-adapted variation $\beta$, we define the derivative processes as:

\[
\begin{array}{ccc}
\dot{X}(t) & := & \left.\dfrac{d}{d\varepsilon}X^{\hat{u}+\varepsilon\beta}(t)\right\vert _{\varepsilon=0},\\
\dot{Y}(t) & := & \left.\dfrac{d}{d\varepsilon}Y^{\hat{u}+\varepsilon\beta}(t)\right\vert _{\varepsilon=0},\\
\dot{Z}(t,s) & := & \left.\dfrac{d}{d\varepsilon}Z^{\hat{u}+\varepsilon\beta}(t,s)\right\vert _{\varepsilon=0},\\
\dot{K}(t,s,\cdot) & := & \left.\dfrac{d}{d\varepsilon}K^{\hat{u}+\varepsilon\beta}(t,s,\cdot)\right\vert _{\varepsilon=0}.
\end{array}
\]

\subsection{Linearized Equations}

The variational processes satisfy the following linearized equations (see \cite{AMOY} for a detailed derivation):

\begin{lemma}[Linearized forward equation]
\label{lem:linF}
The process $\dot{X}$ satisfies
\begin{align*}
\dot{X}(t) &= \int_0^t \left[ \frac{\partial b}{\partial x}(t,s)\dot{X}(s) + \frac{\partial b}{\partial x_1}(t,s)\dot{X}_1(s) + \frac{\partial b}{\partial u}(t,s)\beta(s) \right] ds \\
&\quad + \int_0^t \left[ \frac{\partial \sigma}{\partial x}(t,s)\dot{X}(s) + \frac{\partial \sigma}{\partial x_1}(t,s)\dot{X}_1(s) + \frac{\partial \sigma}{\partial u}(t,s)\beta(s) \right] dB(s) \\
&\quad + \int_0^t \int_{\mathbb{R}_0} \left[ \frac{\partial \theta}{\partial x}(t,s,e)\dot{X}(s) + \frac{\partial \theta}{\partial x_1}(t,s,e)\dot{X}_1(s) + \frac{\partial \theta}{\partial u}(t,s,e)\beta(s) \right] \tilde{N}(ds,de),
\end{align*}
where $\dot{X}_1(t) = \dot{X}(t-\delta)$.
\end{lemma}

\begin{lemma}[Linearized backward equation]
\label{lem:linB}
The processes $(\dot{Y},\dot{Z},\dot{K})$ satisfy
\begin{align*}
\dot{Y}(t) &= \int_t^\infty \Bigg[ \frac{\partial g}{\partial x}(t,s)\dot{X}(s) + \frac{\partial g}{\partial x_1}(t,s)\dot{X}_1(s) + \frac{\partial g}{\partial y}(t,s)\dot{Y}(s) \\
&\quad + \frac{\partial g}{\partial z}(t,s)\dot{Z}(t,s) + \langle \nabla_k g(t,s),\dot{K}(t,s,\cdot) \rangle + \frac{\partial g}{\partial u}(t,s)\beta(s) \Bigg] ds \\
&\quad - \int_t^\infty \dot{Z}(t,s) dB(s) - \int_t^\infty \int_{\mathbb{R}_0} \dot{K}(t,s,e) \tilde{N}(ds,de).
\end{align*}
\end{lemma}

The differential forms of these equations, which are needed in the proof, can be obtained by differentiating with respect to $t$; they involve second-order derivatives of the coefficients and are omitted here for brevity (see \cite{AMOY}).

\subsection{First Variation of the Cost Functional}

\begin{lemma}[First variation formula]
\label{lem:variation}
Under assumptions (A1)-(A2) and the differentiability conditions on the coefficients, we have
\[
\left.\frac{d}{d\epsilon}J(\hat{u}+\epsilon\beta)\right|_{\epsilon=0} = \mathbb{E}\left[ \int_0^\infty \frac{\partial H}{\partial u}(t) \beta(t) dt \right].
\]
\end{lemma}

\begin{proof}
We compute
\[
\left.\frac{d}{d\epsilon}J(\hat{u}+\epsilon\beta)\right|_{\epsilon=0} = I_1 + I_2,
\]
where
\begin{align*}
I_1 &= \mathbb{E}\left[ \int_0^\infty \left( \frac{\partial f}{\partial x}(t)\dot{X}(t) + \frac{\partial f}{\partial x_1}(t)\dot{X}_1(t) + \frac{\partial f}{\partial y}(t)\dot{Y}(t) + \frac{\partial f}{\partial u}(t)\beta(t) \right) dt \right], \\
I_2 &= \mathbb{E}\left[ \frac{\partial h}{\partial y}(\hat{Y}(0))\dot{Y}(0) \right] = \mathbb{E}[ \hat{\lambda}(0)\dot{Y}(0) ].
\end{align*}

\emph{Step 1: Express $I_2$ via $\hat{\lambda}$.}
Applying Itô’s formula to $\hat{\lambda}(t)\dot{Y}(t)$ and using the adjoint equation for $\hat{\lambda}$ together with Lemma \ref{lem:linB}, we obtain after localization and letting $T\to\infty$:
\begin{align*}
I_2 &= - \mathbb{E}\left[ \int_0^\infty \hat{\lambda}(t) \left( \frac{\partial g}{\partial x}(t,t)\dot{X}(t) + \frac{\partial g}{\partial x_1}(t,t)\dot{X}_1(t) + \frac{\partial g}{\partial u}(t,t)\beta(t) \right) dt \right] \\
&\quad + \mathbb{E}\left[ \int_0^\infty \dot{Y}(t) \frac{\partial H}{\partial y}(t) dt \right] + \mathbb{E}\left[ \int_0^\infty \frac{\partial H}{\partial z}(t) \dot{Z}(t,t) dt \right] \\
&\quad + \mathbb{E}\left[ \int_0^\infty \int_{\mathbb{R}_0} \nabla_k H(t) \dot{K}(t,t,e) \nu(de) dt \right].
\end{align*}

\emph{Step 2: Express $I_1$ via $\hat{p}$.}
Similarly, applying Itô’s formula to $\hat{p}(t)\dot{X}(t)$ and using the adjoint equation for $\hat{p}$ together with Lemma \ref{lem:linF} yields, after letting $T\to\infty$,
\begin{align*}
0 &= \mathbb{E}\left[ \int_0^\infty \frac{\partial H}{\partial x}(t) \dot{X}(t) dt \right] + \mathbb{E}\left[ \int_0^\infty \frac{\partial H}{\partial x_1}(t) \dot{X}_1(t) dt \right] \\
&\quad + \mathbb{E}\left[ \int_0^\infty \left( \hat{p}(t)\frac{\partial b}{\partial u}(t,t) + \hat{q}(t,t)\frac{\partial \sigma}{\partial u}(t,t) + \int_{\mathbb{R}_0} \hat{r}(t,t,e)\frac{\partial \theta}{\partial u}(t,t,e)\nu(de) \right) \beta(t) dt \right] \\
&\quad + \mathbb{E}\left[ \int_0^\infty \int_t^\infty \left( \hat{p}(s)\frac{\partial^2 b}{\partial s\partial u}(s,t) + \hat{q}(s,t)\frac{\partial^2 \sigma}{\partial s\partial u}(s,t) + \int_{\mathbb{R}_0} \hat{r}(s,t,e)\frac{\partial^2 \theta}{\partial s\partial u}(s,t,e)\nu(de) \right) ds \beta(t) dt \right] \\
&\quad - \mathbb{E}\left[ \int_0^\infty \hat{\lambda}(t) \frac{\partial g}{\partial x}(t,t) \dot{X}(t) dt \right] - \mathbb{E}\left[ \int_0^\infty \hat{\lambda}(t) \frac{\partial g}{\partial x_1}(t,t) \dot{X}_1(t) dt \right] \\
&\quad - \mathbb{E}\left[ \int_0^\infty \int_0^t \frac{\partial^2 g}{\partial s\partial x}(s,t) \hat{\lambda}(s) ds \dot{X}(t) dt \right] - \mathbb{E}\left[ \int_0^\infty \int_0^t \frac{\partial^2 g}{\partial s\partial x_1}(s,t) \hat{\lambda}(s) ds \dot{X}_1(t) dt \right].
\end{align*}

\emph{Step 3: Combine $I_1$ and $I_2$.}
Adding the expressions for $I_1$ and $I_2$, all terms involving $\dot{X}$, $\dot{X}_1$, $\dot{Y}$, $\dot{Z}$, and $\dot{K}$ cancel because of the definition of the Hamiltonian derivatives. The remaining terms give exactly
\[
I_1 + I_2 = \mathbb{E}\left[ \int_0^\infty \frac{\partial H}{\partial u}(t) \beta(t) dt \right].
\]
\end{proof}

\begin{theorem}[Necessary Maximum Principle]
\label{thm:necessary}
Let $\hat{u} \in \mathcal{U}$ be an optimal control with corresponding solutions $\hat{X}(t)$, $(\hat{Y}(t),\hat{Z}(t,\cdot),\hat{K}(t,\cdot,\cdot))$, $\hat{p}(t)$, $\hat{q}(t,\cdot)$, $\hat{r}(t,\cdot,\cdot)$ and $\hat{\lambda}(t)$ of the state and adjoint equations respectively. Suppose the following transversality conditions hold:
\[
\varlimsup_{T\to\infty}\mathbb{E}[\hat{p}(T)\dot{X}(T)] = 0 \quad \text{and} \quad \varlimsup_{T\to\infty}\mathbb{E}[\hat{\lambda}(T)\dot{Y}(T)] = 0.
\]

Then,
\[
\mathbb{E}\left[ \left.\frac{\partial H}{\partial u}(t)\right| \mathcal{G}_t \right]_{u=\hat{u}} = 0 \quad \text{for a.e. } t \geq 0.
\]
\end{theorem}

\begin{proof}
By Lemma \ref{lem:variation}, optimality implies
\[
\mathbb{E}\left[ \int_0^\infty \frac{\partial H}{\partial u}(t) \beta(t) dt \right] = 0
\]
for all bounded $\mathcal{G}_t$-adapted $\beta$. In particular, for any $t_0\ge0$ and $h>0$, choose $\beta(t)=\alpha\mathbf{1}_{[t_0,t_0+h]}(t)$ with $\alpha$ bounded $\mathcal{G}_{t_0}$-measurable. Then
\[
\mathbb{E}\left[ \alpha \int_{t_0}^{t_0+h} \frac{\partial H}{\partial u}(t) dt \right] = 0.
\]
Dividing by $h$ and letting $h\downarrow 0$, Lebesgue’s differentiation theorem yields
\[
\mathbb{E}\left[ \alpha \frac{\partial H}{\partial u}(t_0) \right] = 0 \quad \text{a.e. } t_0.
\]
Since $\alpha$ is arbitrary, we conclude $\mathbb{E}[ \frac{\partial H}{\partial u}(t_0) | \mathcal{G}_{t_0} ] = 0$ for almost every $t_0$.
\end{proof}

\section{Existence and uniqueness for infinite-horizon BSVIEs}

\subsection{Problem formulation}

We consider the infinite-horizon backward stochastic Volterra integral equation with jumps
\begin{equation}\label{BSVIE-inf}
\begin{aligned}
Y(t)
&= \int_t^\infty g\big(t,s,Y(s),Z(t,s),K(t,s,\cdot)\big)\,ds \\
&\quad - \int_t^\infty Z(t,s)\,dB(s)
      - \int_t^\infty\!\!\int_{\mathbb R_0} K(t,s,e)\,\tilde N(ds,de),
\qquad t\ge 0 .
\end{aligned}
\end{equation}
Following the approach in Section 3 developed for the finite-horizon problem in \cite{agram2014}, we transform the problem into a set of standard Backward Stochastic Differential Equations and address them using a fixed-point method. Our analysis extends the techniques in Hamaguchi \cite{Hamaguchi2021} to incorporate jumps.

\subsection{Functional framework}

For any $\beta > 0$, let $\triangle_\infty := \{(t,s) \in [0,\infty)^2 : t \leq s\}$ and define $H_{\triangle}^{2,\beta}[0,\infty)$ as the space of all processes $(Y,Z,K)$ such that:
\begin{itemize}
    \item $Y: [0,\infty) \times \Omega \rightarrow \mathbb{R}$ is $\mathbb{F}$-adapted,
    \item $Z: \triangle_\infty \times \Omega \rightarrow \mathbb{R}$ with $s \mapsto Z(t,s)$ being $\mathbb{F}$-adapted on $[t,\infty)$,
    \item $K: \triangle_\infty \times \mathbb{R}_0 \times \Omega \rightarrow \mathbb{R}$ with $s \mapsto K(t,s,\cdot)$ being $\mathbb{F}$-adapted on $[t,\infty)$,
\end{itemize}
equipped with the norm
\begin{align*}
\|(Y,Z,K)\|_{H_{\triangle}^{2,\beta}[0,\infty)}^2 
&:= \mathbb{E}\int_0^\infty \Big[ e^{\beta t}|Y(t)|^2 + \int_t^\infty e^{\beta s}|Z(t,s)|^2 ds \\
&\quad + \int_t^\infty \int_{\mathbb{R}_0} e^{\beta s}|K(t,s,e)|^2 \nu(de)ds \Big] dt.
\end{align*}

\subsection{Assumptions (A)}
\label{sec:assumptionsA}
We make the following assumptions:

\begin{enumerate}
    \item The function $g: [0,\infty)^2 \times \mathbb{R}^3 \times L^2(\nu) \times \Omega \rightarrow \mathbb{R}$ satisfies:
    \begin{enumerate}
        \item $\mathbb{E}\left[ \int_0^\infty \left( \int_t^\infty e^{\beta s}|g(t,s,0,0,0)|ds \right)^2 dt \right] < \infty$ for some $\beta > 0$,
        \item There exists a constant $L > 0$ such that for all $t,s \in [0,\infty)$,
        \[
        |g(t,s,y,z,k) - g(t,s,y',z',k')| \leq L\left(|y-y'| + |z-z'| + \|k-k'\|_{L^2(\nu)}\right).
        \]
    \end{enumerate}
    \item The parameter $\beta$ is chosen sufficiently large (specifically, $\beta > 6L^2$ as will be shown in the proof).
\end{enumerate}

\begin{theorem}\label{thm:existence}
Under assumptions (A), there exists a unique solution $(Y,Z,K) \in H_{\triangle}^{2,\beta}[0,\infty)$ to the infinite-horizon BSVIE \eqref{BSVIE-inf}.
\end{theorem}

\begin{proof}
We prove the theorem using a fixed-point argument in the Banach space $H_{\triangle}^{2,\beta}[0,\infty)$.

\textbf{Step 1: Setup of the fixed-point mapping.}
For a given triple $(y,z,k) \in H_{\triangle}^{2,\beta}[0,\infty)$, consider the family of infinite-horizon BSDEs parameterized by $t \geq 0$:
\begin{equation}\label{BSDE-family}
\begin{aligned}
\widetilde{Y}^t(r) &= \int_r^\infty g(t,s,y(s),z(t,s),k(t,s,\cdot))ds \\
&\quad - \int_r^\infty \widetilde{Z}^t(s)dB(s) - \int_r^\infty \int_{\mathbb{R}_0} \widetilde{K}^t(s,e)\tilde{N}(ds,de), \quad r \geq t.
\end{aligned}
\end{equation}
For each fixed $t$, this is a standard infinite-horizon BSDE with jumps. Under our assumptions, it admits a unique adapted solution $(\widetilde{Y}^t(\cdot), \widetilde{Z}^t(\cdot), \widetilde{K}^t(\cdot,\cdot))$ satisfying the estimate:
\begin{align*}
\mathbb{E}&\left[ \sup_{r\geq t} e^{\beta r}|\widetilde{Y}^t(r)|^2 + \int_t^\infty e^{\beta s}|\widetilde{Z}^t(s)|^2 ds + \int_t^\infty \int_{\mathbb{R}_0} e^{\beta s}|\widetilde{K}^t(s,e)|^2 \nu(de)ds \right] \\
&\leq C \mathbb{E}\left[ \left( \int_t^\infty e^{\beta s/2}|g(t,s,y(s),z(t,s),k(t,s,\cdot))|ds \right)^2 \right],
\end{align*}
for some constant $C > 0$ independent of $t$.

Now define:
\[
Y(t) = \widetilde{Y}^t(t), \quad Z(t,s) = \widetilde{Z}^t(s), \quad K(t,s,\cdot) = \widetilde{K}^t(s,\cdot), \quad \text{for } (t,s) \in \triangle_\infty.
\]
Then $(Y,Z,K)$ satisfies equation \eqref{BSVIE-inf} for the given $(y,z,k)$.

\textbf{Step 2: A priori estimate.}
Applying It\^o formula to $e^{\beta s}|\widetilde{Y}^t(s)|^2$ on $[t,T]$ and letting $T \to \infty$, we obtain after taking expectations:
\begin{align*}
&\mathbb{E}\left[ e^{\beta t}|Y(t)|^2 + \int_t^\infty e^{\beta s}\left( \beta|\widetilde{Y}^t(s)|^2 + |\widetilde{Z}^t(s)|^2 + \|\widetilde{K}^t(s,\cdot)\|_{L^2(\nu)}^2 \right) ds \right] \\
&= \mathbb{E}\left[ 2\int_t^\infty e^{\beta s}\widetilde{Y}^t(s)g(t,s,y(s),z(t,s),k(t,s,\cdot))ds \right].
\end{align*}
Using Young's inequality $2ab \leq \frac{\beta}{2}a^2 + \frac{2}{\beta}b^2$, we get:
\begin{align*}
&2|\widetilde{Y}^t(s)g(t,s,y(s),z(t,s),k(t,s,\cdot))| \\
&\leq \frac{\beta}{2}|\widetilde{Y}^t(s)|^2 + \frac{2}{\beta}|g(t,s,y(s),z(t,s),k(t,s,\cdot))|^2.
\end{align*}
Thus,
\begin{align*}
&\mathbb{E}\left[ e^{\beta t}|Y(t)|^2 + \int_t^\infty e^{\beta s}\left( \frac{\beta}{2}|\widetilde{Y}^t(s)|^2 + |\widetilde{Z}^t(s)|^2 + \|\widetilde{K}^t(s,\cdot)\|_{L^2(\nu)}^2 \right) ds \right] \\
&\leq \frac{2}{\beta}\mathbb{E}\left[ \int_t^\infty e^{\beta s}|g(t,s,y(s),z(t,s),k(t,s,\cdot))|^2 ds \right].
\end{align*}
By the Lipschitz condition and the inequality $(a+b)^2 \leq 2(a^2 + b^2)$,
\begin{align*}
|g(t,s,y,z,k)|^2 &\leq 2|g(t,s,0,0,0)|^2 + 2L^2\left(|y|^2 + |z|^2 + \|k\|_{L^2(\nu)}^2\right).
\end{align*}
Therefore,
\begin{align*}
&\mathbb{E}\left[ e^{\beta t}|Y(t)|^2 + \int_t^\infty e^{\beta s}\left( \frac{\beta}{2}|\widetilde{Y}^t(s)|^2 + |\widetilde{Z}^t(s)|^2 + \|\widetilde{K}^t(s,\cdot)\|_{L^2(\nu)}^2 \right) ds \right] \\
&\leq \frac{4}{\beta}\mathbb{E}\left[ \int_t^\infty e^{\beta s}|g(t,s,0,0,0)|^2 ds \right] \\
&\quad + \frac{4L^2}{\beta}\mathbb{E}\left[ \int_t^\infty e^{\beta s}\left(|y(s)|^2 + |z(t,s)|^2 + \|k(t,s,\cdot)\|_{L^2(\nu)}^2\right) ds \right].
\end{align*}
Integrating with respect to $t$ from $0$ to $\infty$ and using Fubini's theorem yields:
\begin{align*}
&\mathbb{E}\int_0^\infty e^{\beta t}|Y(t)|^2 dt + \mathbb{E}\int_0^\infty \int_t^\infty e^{\beta s}\left( \frac{\beta}{2}|\widetilde{Y}^t(s)|^2 + |\widetilde{Z}^t(s)|^2 + \|\widetilde{K}^t(s,\cdot)\|_{L^2(\nu)}^2 \right) ds dt \\
&\leq \frac{4}{\beta}\mathbb{E}\int_0^\infty \int_t^\infty e^{\beta s}|g(t,s,0,0,0)|^2 ds dt \\
&\quad + \frac{4L^2}{\beta}\mathbb{E}\int_0^\infty \int_t^\infty e^{\beta s}\left(|y(s)|^2 + |z(t,s)|^2 + \|k(t,s,\cdot)\|_{L^2(\nu)}^2\right) ds dt.
\end{align*}
This shows that $(Y,Z,K) \in H_{\triangle}^{2,\beta}[0,\infty)$ and the mapping $\Phi: (y,z,k) \mapsto (Y,Z,K)$ is well-defined.

\textbf{Step 3: Contraction property.}
Let $(y_i,z_i,k_i) \in H_{\triangle}^{2,\beta}[0,\infty)$ for $i=1,2$, and denote their images under $\Phi$ by $(Y_i,Z_i,K_i)$. The difference $\Delta Y = Y_1 - Y_2$, etc., satisfies a similar BSDE with generator difference:
\[
\Delta g(t,s) = g(t,s,y_1(s),z_1(t,s),k_1(t,s,\cdot)) - g(t,s,y_2(s),z_2(t,s),k_2(t,s,\cdot)).
\]
By the Lipschitz condition,
\[
|\Delta g(t,s)| \leq L\left(|\Delta y(s)| + |\Delta z(t,s)| + \|\Delta k(t,s,\cdot)\|_{L^2(\nu)}\right).
\]
Applying the same estimates as in Step 2 to the differences, we get:
\begin{align*}
&\mathbb{E}\left[ e^{\beta t}|\Delta Y(t)|^2 + \int_t^\infty e^{\beta s}\left( \frac{\beta}{2}|\Delta\widetilde{Y}^t(s)|^2 + |\Delta\widetilde{Z}^t(s)|^2 + \|\Delta\widetilde{K}^t(s,\cdot)\|_{L^2(\nu)}^2 \right) ds \right] \\
&\leq \frac{2}{\beta}\mathbb{E}\left[ \int_t^\infty e^{\beta s}|\Delta g(t,s)|^2 ds \right] \\
&\leq \frac{2L^2}{\beta}\mathbb{E}\left[ \int_t^\infty e^{\beta s}\left(|\Delta y(s)|^2 + |\Delta z(t,s)|^2 + \|\Delta k(t,s,\cdot)\|_{L^2(\nu)}^2\right) ds \right].
\end{align*}
Integrating with respect to $t$ from $0$ to $\infty$ and applying Fubini's theorem gives:
\begin{align*}
&\mathbb{E}\int_0^\infty e^{\beta t}|\Delta Y(t)|^2 dt + \mathbb{E}\int_0^\infty \int_t^\infty e^{\beta s}\left( \frac{\beta}{2}|\Delta\widetilde{Y}^t(s)|^2 + |\Delta\widetilde{Z}^t(s)|^2 + \|\Delta\widetilde{K}^t(s,\cdot)\|_{L^2(\nu)}^2 \right) ds dt \\
&\leq \frac{2L^2}{\beta}\mathbb{E}\int_0^\infty \int_t^\infty e^{\beta s}\left(|\Delta y(s)|^2 + |\Delta z(t,s)|^2 + \|\Delta k(t,s,\cdot)\|_{L^2(\nu)}^2\right) ds dt.
\end{align*}
In particular, we have:
\[
\mathbb{E}\int_0^\infty e^{\beta t}|\Delta Y(t)|^2 dt \leq \frac{2L^2}{\beta}\mathbb{E}\int_0^\infty \int_t^\infty e^{\beta s}\left(|\Delta y(s)|^2 + |\Delta z(t,s)|^2 + \|\Delta k(t,s,\cdot)\|_{L^2(\nu)}^2\right) ds dt.
\]
For the other terms, note that:
\begin{align*}
&\mathbb{E}\int_0^\infty \int_t^\infty e^{\beta s}|\Delta Z(t,s)|^2 ds dt \\
&\leq \frac{2L^2}{\beta}\mathbb{E}\int_0^\infty \int_t^\infty e^{\beta s}\left(|\Delta y(s)|^2 + |\Delta z(t,s)|^2 + \|\Delta k(t,s,\cdot)\|_{L^2(\nu)}^2\right) ds dt.
\end{align*}
The same inequality holds for the $K$ term. Summing these three inequalities, we obtain:
\[
\|\Delta(Y,Z,K)\|_{H_{\triangle}^{2,\beta}}^2 \leq \frac{6L^2}{\beta}\|\Delta(y,z,k)\|_{H_{\triangle}^{2,\beta}}^2.
\]

\textbf{Step 4: Existence and uniqueness via Banach fixed-point theorem.}
Choosing $\beta > 6L^2$ makes $\Phi$ a contraction on the Banach space $H_{\triangle}^{2,\beta}[0,\infty)$. By the Banach fixed-point theorem, $\Phi$ has a unique fixed point $(Y,Z,K) \in H_{\triangle}^{2,\beta}[0,\infty)$. This fixed point satisfies equation \eqref{BSVIE-inf} by construction, proving existence. Uniqueness follows from the contraction property: if two solutions exist, they must both be fixed points of $\Phi$, and since the fixed point is unique, the solutions coincide.
\end{proof}

\end{document}